\documentclass[11pt]{amsart}
\usepackage{amsfonts}
\usepackage{amssymb}
\usepackage[T1]{fontenc}
\usepackage{amsmath}
\usepackage{dsfont}
\usepackage{color}
\usepackage{graphicx}
\newtheorem{theorem}{Theorem}[section]
\newtheorem{lemma}[theorem]{Lemma}
\newtheorem{proposition}[theorem]{Proposition}
\newtheorem{corollary}[theorem]{Corollary}

\theoremstyle{definition}
\newtheorem{df}{Definition}
\newtheorem{example}[df]{Example}
\newtheorem{remark}[df]{Remark}
\newtheorem{problem}[df]{Problem}

\newcommand{\N}{\mathbb N}

\newcommand{\R}{\mathbb R}

\newcommand{\eL}{\mathcal{L}}

\newcommand{\F}{\mathcal{F}}

\makeatletter
\@namedef{subjclassname@2020}{%
  \textup{2020} Mathematics Subject Classification}
\makeatother

\subjclass[2020]{08B20,40A05,28A35} 
\keywords{Banach limits, porosity, algebrability, measure of families of sequences}

\begin{document}
\author{Piotr Nowakowski}
\address{Faculty of Mathematics and Computer Science, University of \L \'{o}d\'{z},
Banacha 22, 90-238 \L \'{o}d\'{z}, Poland \\
ORCID: 0000-0002-3655-4991}
\email{piotr.nowakowski@wmii.uni.lodz.pl}

\title{How large is the space of almost convergent sequences?}
\date{}

\begin{abstract}
We consider the subspaces $c$, $\widehat{c}$, $S$ of $\ell^\infty$, where $\widehat{c}$ consists of almost convergent sequences, and $S$ consists of sequences whose arithmetic means of consecutive terms are convergent. We know that $c\subset \widehat{c} \subset S$. We examine the largeness of $c$ in $\widehat{c}$, $\widehat{c}$ in $S$ and $S$ in $\ell^\infty$. We will do it from the viewpoints of porosity, algebrability and measure. 
\end{abstract}

\maketitle

\section{Introduction}
We say that a linear functional $L\colon \ell^{\infty} \to \R$ is a Banach limit if for any sequence $(x_n) \in \ell^{\infty}$ the following conditions hold (see \cite{S}):

1) $\forall_{n\in\N}\,\, x_n \geq 0 \Rightarrow L((x_n))\geq 0;$

2) $L(T((x_n))) = L((x_n))$, where $T((x_n)) = (x_2,x_3, \dots)$;

3) $L((1,1,\dots)) = 1$.\\
When we write that some number $b$ is a Banach limit of a sequence $(x_n)$, we mean that there is a Banach limit $L\colon \ell^{\infty} \to \R$ such that $L((x_n))=b$. Banach limits are widely investigated by many authors in various areas of mathematics (e.g. \cite{AM}, \cite{P}, \cite{FL}). Recently, also the monograph \cite{DN} concerning Banach limits was published.

It can be easily proved that if a sequence is convergent, then every Banach limit on this sequence is equal to the classical limit of the sequence. It is natural to ask about the existence of non-convergent sequences with such a property, that is,  which have a unique Banach limit. The answer is positive. For example, every Banach limit on the sequence $(-1)^n$ is equal to $0$. The main goal of this paper is to check how large is the set of such sequences with a unique Banach limit.

Define $$\widehat{c}: = \{x \in \ell^{\infty} \colon \exists_{s \in \R} \; L(x) = s \mbox{ for any Banach limit } L\}$$ and 
$$S:=\{x \in \ell^{\infty} \colon \lim\limits_{n\to \infty} \frac{x_1+\dots+x_n}{n} \mbox{ exists}\}.$$
Both sets are linear subspaces of $\ell^\infty$. 

The existence of sequences with unique Banach limits was observed by Lorentz in \cite{L}. They are called almost convergent sequences. There are many researchers working on almost convergent sequences (see e.g. \cite{CG}, \cite{D}, \cite{MV})

Lorentz proved the following theorem.
\begin{theorem} \label{t1} \cite{L}
Let $x \in \ell^{\infty}$, $s \in \R$. The following conditions are equivalent:

i) $x \in \widehat{c}$ and $L(x) = s$ for every Banach limit $L$;

ii) $\lim\limits_{n \to \infty} \frac{x_j+\dots+x_{j+n-1}}{n} = s$ uniformly with respect to $j$. 
\end{theorem}
\begin{remark}
Condition ii) is equivalent to:
$$ \lim\limits_{n \to \infty}\left(\sup_j \left\{\frac{x_j+\dots+x_{j+n-1}}{n}\right\}\right) = \lim\limits_{n \to \infty}\left(\inf_j \left\{\frac{x_j+\dots+x_{j+n-1}}{n}\right\}\right)=s.$$
\end{remark}
\begin{corollary}
$\widehat{c} \subset S$ and for every $x \in \widehat{c}$ and every Banach limit $L$ we have $$L(x) = \lim\limits_{n\to \infty} \frac{x_1+\dots+x_n}{n}.$$
\end{corollary}

One may ask, if the inclusion $\widehat{c} \subset S$ is proper. The following example shows that the answer is positive.

\begin{example}
Let $x = (x_n)$ be defined in the following way: put $m_0 := 0$ and $m_j := m_{j-1}+j+2^j$ for $j \in \N$. Then for $n \in \N$ we define
$$x_n := \left\{ \begin{array}{ccc}
0\;\text{ if }\; n \in \{m_{j-1}+1, m_{j-1}+2, \dots, m_{j-1}+j\} \\ 
1 \;\text{ for the remaining }n,%
\end{array}%
\right. $$
that is, $x=(0,1,1,0,0,1,1,1,1,0,0,0,1,1,1,1,1,1,1,1,0\dots)$. Then in the set $\{m_{j-1}+1, \dots, m_j\}$ we have exactly $j$ zeros and $2^j$ ones. If $n \in \{m_{j-1}+1, \dots, m_j\}$, then
$$1 \geq \frac{x_1+x_2+\dots+x_n}{n} \geq \frac{\sum_{i=1}^{j-1}2^i}{m_{j-1} + j} = \frac{\sum_{i=1}^{j-1}2^i}{\sum_{i=1}^{j-1}2^i + \sum_{i=1}^j i} = \frac{2^j-2}{2^{j}-2+\frac{j^2+j}{2}}. $$
Observe that if $n$ tends to infinity, so does $j$. Thus, the number $\frac{2^j-2}{2^{j}-2+\frac{j^2+j}{2}}$ converges to $1$ as $n$ tends to infinity.
Hence $\lim\limits_{n\to \infty} \frac{x_1+\dots+x_n}{n} = 1,$ so $x \in S$. 
Similarly, for any $j \in \N$, $\lim\limits_{n\to \infty} \frac{x_j+\dots+x_{n-1+j}}{n} = 1.$
On the other hand, for any $n \in \N$ we can find $j \in \N$ such that $\frac{x_j+\dots+x_{n-1+j}}{n} = 0.$ So, 
$\lim\limits_{n \to \infty}\left(\inf_j \left\{\frac{x_j+\dots+x_{j+n-1}}{n}\right\}\right)=0.$ Thus, despite the fact that for all  $j \in \N$, $\lim\limits_{n\to \infty} \frac{x_j+\dots+x_{n-1+j}}{n} = 1,$ this convergence is not uniform with respect to $j$. By Theorem \ref{t1}, $x \notin \widehat{c}$. 
Therefore, $\widehat{c} \subsetneq S$.
\end{example}

Consider the spaces $c$, $c_0$, $\widehat{c_0}$ and $S_0$ contained in $\ell^{\infty}$ of sequences which are convergent, convergent to zero, almost convergent to $0$, with arithmetic means convergent to zero, respectively. We have
$c_0 \subsetneq c \subsetneq \widehat{c}$, $c_0 \subsetneq \widehat{c_0}\subsetneq S_0 \subsetneq S$ and $\widehat{c_0} \subsetneq \widehat{c}$.

\begin{problem}
How large is the space $c$ in $\widehat{c}$ and $\widehat{c}$ in $S$? Similarly, how large is $c_0$ in $\widehat{c_0}$ and $\widehat{c_0}$ in $S_0$?
\end{problem}

In the paper, we try to solve this problem looking at it from different viewpoints. In section 2. we check whether the considered spaces are porous in each other. In section 3. we examine the algebrability of the spaces $\widehat{c} \setminus c$, $S \setminus \widehat{c}$ and $\ell^{\infty} \setminus S$. In section 4. we focus on measure of the considered families of sequences.

\section{Porosity of considered spaces}
Let us recall the notions of porous sets in a metric space (see \cite{Za}, \cite{Za1}).
Let $(X,d)$ be a metric space. For $x\in X$ and $r > 0$ we write 
$$B(x,r) = \{y\in X\colon d(x,y) < r\}.$$
For $E\subset X$, $x \in X$ and $R >0$, we set $\gamma(x,R,E):=\sup\{r > 0 \colon \exists\,{z \in X} \,\,(B(z,r) \subset B(x,R)\setminus E)\}$. 
The porosity of $E$ at $x$ is defined as
\begin{equation*} \label{por}
p(E,x):=2 \limsup\limits_{r \to 0^+}  \frac{\gamma(x,r,E)}{r}.
\end{equation*}
We say that a set $E$ is porous if its porosity is positive at each $x \in E$, and $E$ is strongly porous if its porosity is equal to $1$ at each $x \in E$. It is well known that every porous set is nowhere dense. We may also consider the lower porosity, where we replace $\limsup$ in the definition of the porosity by $\liminf$. We then define lower porosity and strong lower porosity analogously as porosity and strong porosity.

We may look at porous sets as "small sets". In this section we check porosity properties of the considered families of sequences. Examining whether some normed space is porous in another one is not a new idea. For example, this has been done recently in \cite{B}.

First, we need to prove the following easy proposition.
\begin{proposition} \label{dom}
The spaces $\widehat{c}, \widehat{c_0}, S$ and $S_0$ are closed in $\ell^\infty$.
\end{proposition}
\begin{proof}
Let $(x^n)_{n\in\N}$ be a convergent sequence in $\widehat{c}$. Denote by $x$ its limit in $\ell^\infty$. Let $L\colon \ell^\infty\to \R$ be a Banach limit. For $n\in \N$ denote by $y_n$ a unique Banach limit of a sequence $x^n$. Then also $L(x^n)=y_n$ for any $n \in \N$. Since $L$ is bounded linear operator it is continuous, thus $$L(x) = \lim\limits_{n\to \infty}L(x^n) = \lim\limits_{n\to \infty}L(y_n).$$ 
So, $x$ has a unique Banach limit, that is, $x \in \widehat{c}$. Therefore, $\widehat{c}$ is closed. We may similarly show the closedness of $\widehat{c_0}$.

Let $(x^n)=((x^n_i)_{i\in\N})_{n\in\N}$ be a convergent sequence in $S$. Denote by $x=(x_i)$ its limit in $\ell^\infty$. For any $n,i\in\N$ put
$s^n_i = \frac{x^n_1+\dots+x^n_i}{i}$ and
$s_i = \frac{x_1+\dots+x_i}{i}$. Since $\lim\limits_{n\to \infty} x^n_i = x_i$ for any $n\in\N$, we also have $\lim\limits_{n\to \infty} s^n_i = s_i$ for all $n$. By closedness of $c$, we have that a limit of the convergent sequence $(s^n)$ (that is, $(s_i)$) is also convergent, which proves that $x\in S$. Hence $S$ is closed. The proof for $S_0$ is analogous.
\end{proof}

Before proving the main result of this section, we need some definitions and theorems from the paper \cite{S}.

For the rest of this section we assume that $(X,||\cdot||)$ is a normed space. For $M\subset X$ we denote by $conv(M)$ a convex hull of $M$.

We say that $M\subset X$ is c-porous if for any $x\in X$ and every $r >0$, there are $y\in B(x,r)$ and non-zero continuous linear functional $\phi\colon X \to \R$ such that
$$\{z \in X\colon \phi(z) > \phi(y)\} \cap M = \emptyset.$$

\begin{proposition}\cite[Proposition 2.8.]{S} \label{2.8S}
If $M \subset X$ is c-porous, then for every $R>0$, $x \in X$ and $\alpha \in (0,1)$, there exists $y \in X$ such that $||y-x|| = R$ and $B(y,\alpha R)\cap M = \emptyset$.
\end{proposition}
\begin{corollary}\label{wniosek por}
If $M \subset X$ is c-porous, then it is strongly lower porous. 
\end{corollary}
\begin{proof}
Let $x \in M$, $r >0$ and $\alpha \in (0,1)$. By Proposition \ref{2.8S}, there exists $y \in X$ such that $||y-x|| = \frac{r}{2}$ and $B(y,\alpha \frac{r}{2})\cap M = \emptyset$. Thus, 
$\gamma(x,r,M) \geq \alpha \frac{r}{2},$ and by arbitrariness of $\alpha$, we get $\gamma(x,r,M) \geq \frac{r}{2}.$ Since also $\gamma(x,r,M) \leq \frac{r}{2}$ (because $x\in M$), we have $\gamma(x,r,M) = \frac{r}{2}$. Hence lower porosity of $M$ at $x$ is equal to $$2\liminf\limits_{r\to 0^+} \frac{\gamma(x,r,E)}{r} = 1.$$ By the arbitrariness of $x$, $M$ is strongly lower porous.
\end{proof}
\begin{proposition}\cite[Proposition 2.5]{S} \label{prop2.5S}
$M \subset X$ is c-porous if and only if $conv(M)$ is nowhere dense.
\end{proposition}
The following lemma is a mathematical folklore. We present its short proof for a reader's convenience. 
\begin{lemma} \label{lemacik}
Let $M \subsetneq X$ be a linear subspace. If $M$ is closed, it is nowhere dense.
\end{lemma}
\begin{proof}
Assume that $M$ is closed. Suppose that $M$ is not nowhere dense. Then there exist $x\in X$ and $r>0$ such that $B(x,r) \subset M$. Since $M$ is linear, then also $B(0,r) \subset M$. Let $y \in X$. Take $R > 0$ such that $\frac{||y||}{R} < r$. Then $\frac{||y||}{R} \in M$, and so $y = R\cdot \frac{||y||}{R} \in M$. Therefore, $M=X$, a contradiction. Finally, $M$ is nowhere dense.
\end{proof}
\begin{theorem} \label{tp1} 
The following conditions hold

\begin{itemize}

\item $c$ is strongly lower porous in $\widehat{c}$;

\item $c_0$ is strongly lower porous in $\widehat{c_0}$.

\item $\widehat{c}$ is strongly lower porous in $S$;

\item $\widehat{c_0}$ is strongly lower porous in $S_0$;

\item $S$ is strongly lower porous in $\ell^\infty$.

\end{itemize}
\end{theorem}
\begin{proof}
We will only prove the first assertion, because the reasoning for the rest is similar (we need to use Proposition \ref{dom}). Since $c$ is a linear subspace, it is convex. Because it is closed, by Lemma \ref{lemacik}, it is nowhere dense (in $\widehat{c}$). Using Proposition \ref{prop2.5S}, we get that $c$ is c-porous in $\widehat{c}$. By Proposition \ref{2.8S}, $c$ is strongly lower porous in $\widehat{c}$.
\end{proof}

\section{Algebrability}
One of the reasons to call a space large may be to find some big structure inside like an algebra generated by many elements. Such a reasoning has appeared already in \cite{LM} and later in papers of Gurariy \cite{G1}, \cite{G2}. Following this way of thinking, in \cite{AGS} and \cite{AS} the notions of  lineability, spaceability and algebrability were introduced. Let $\kappa$ be a cardinal number.
We say that
\begin{itemize}
\item a subset $A$ of a vector space $\eL$ is $\kappa$-lineable if $A\cup \{0\}$ contains a $\kappa$-dimensional vector space;
\item a subset $A$ of a Banach space $\eL$ is spaceable if $A\cup \{0\}$ contains an infinite dimensional closed vector space;
\item a subset $A$ of a linear commutative algebra $\eL$ is $\kappa$-algebrable if $A\cup \{0\}$ contains a $\kappa$-generated algebra $B$, that is, the minimal number of generators of $B$ has cardinality $\kappa$.
\end{itemize}
In \cite{BG} there was introduced a strengthened notion of algebrability. We say that a subset $A$ of a linear commutative algebra $\eL$ is strongly $\kappa$-algebrable if $A\cup \{0\}$ contains a $\kappa$-generated algebra which is isomorphic to a free algebra. It is an easy observation that strong $\kappa$-algebrability implies $\kappa$-algebrability, which implies $\kappa$-lineability.
It is worth mentioning that in the last 20 years there appear a lot of interesting results concerning algebrability (e.g. \cite{AMS}, \cite{GGMS}, \cite{GPS}, \cite{GKP}, \cite{SW}).

In \cite{BBF} there were proved two crucial results which we will use in our paper. But first, we need a notion of an exponential-like function. We say that $f\colon \R \to \R$ is exponential like (of rank $m$) if $f(x) =\sum_{i=1}^{m} \alpha_i e^{\beta_i x}$ for some distinct non-zero real numbers $\beta_1, \dots \beta_m$ and some $\alpha_1,\dots, \alpha_m \in \R\setminus\{0\}$.
\begin{theorem}\cite{BBF} \label{texp}
Let $\F\subset \R^{\N}$ and assume that there exists a sequence $z = (z_1, z_2,\dots)\in \F$ such that $(f(z_1),f(z_2),\dots)\in \F \setminus \{(0,0,\dots)\}$ for every exponential-like function $f\colon \R \to \R.$ Then $\F$ is strongly $\mathfrak{c}$-algebrable.
\end{theorem}
Originally, $F$ was considered as a subset of $\R^{[0,1]}$ not $\R^{\N}$, but this replacement does not change the proof. It has been already pointed out in \cite{BBFG}.
\begin{lemma}\label{lemexp} \cite{BBF}
For any $n \in \N$, any exponential-like function $f\colon \R\to \R$ of rank $m$ and any $c \in \R$, the preimage $f^{-1}(\{c\})$ has at most $m$ elements.
\end{lemma}
Again, originally the domain of $f$ was $[0,1]$, but we can easily replace it with $\R$. Such version of this lemma was used for example in \cite{BBFG}.

\begin{theorem}
The family $\widehat{c} \setminus c$ is strongly $\mathfrak{c}$-algebrable.
\end{theorem}
\begin{proof}
First, put $$(a_n) := (1,2,\frac{3}{2},\frac{5}{4},\frac{7}{4},\frac{9}{8},\frac{11}{8},\frac{13}{8},\frac{15}{8},\dots).$$
Now, put $$(b_n):= (a_1,a_1,a_2,a_1,a_2,a_3,a_1,a_2,a_3,a_4,\dots).$$
Let $z$ be a sequence defined in the following way: put $m_1:=1$ and $m_j:=m_{j-1}+j$ for $j >1$. For $n \in \N$ we define
$$z_n := \left\{ \begin{array}{ccc}
b_n\;\text{ if }\; n=m_j \\ 
0 \;\text{ for the remaining }n,%
\end{array}%
\right. $$
that is, $$z=(1,0,1,0,0,2,0,0,0,1,0,0,0,0,2,0,0,0,0,0,\frac{3}{2},\dots).$$

Let $f\colon\R\to \R$ be an exponential-like function. We will show that $$f(z):=(f(z_1),f(z_2),\dots)\in \widehat{c} \setminus c.$$ By Lemma \ref{lemexp}, we know that every value of $f$ can be obtained only on finitely many arguments. Since $z$ is a sequence which admits infinitely many values and each of them appears infinitely many times, also $f(z)$ has this property. Hence the sequence $f(z)$ is not convergent. 

Denote by $L$ and $M$, respectively, the minimal and the maximal values of $f$ on $[1,2]$. They exist by Weierstrass theorem. Take $k,j \in \N$ and $n \in \{m_{j},m_j+1,\dots, m_{j+1}-1\}.$ Then in the finite sequence $(z_k, z_{k+1}, \dots, z_{k+n-1})$, we have at least $\frac{j(j-1)}{2}$ zeros.
If $f(0) \leq 0$, we have 
$$\frac{f(z_k)+f(z_{k+1})+\dots +f(z_{k+n-1})}{n} \leq \frac{jM+\frac{j(j-1)}{2}f(0)}{m_j}=
 \frac{jM+\frac{j(j-1)}{2}f(0)}{\frac{j(j+1)}{2}}=\frac{2M+f(0)j-f(0)}{j+1}$$
and if $f(0)>0$, then 
$$\frac{f(z_k)+f(z_{k+1})+\dots f(z_{k+n-1})}{n} \leq 
 \frac{jM+\frac{j(j+1)}{2}f(0)}{\frac{j(j+1)}{2}}=\frac{2M}{j+1}+f(0).$$
The both sequences converge to $f(0)$, when $n$ (and thus also $j$) tends to infinity.
Similarly, if $f(0) \geq 0$, we have
 $$\frac{f(z_k)+f(z_{k+1})+\dots f(z_{k+n-1})}{n} \geq \frac{jL+\frac{j(j-1)}{2}f(0)}{m_{j+1}}=
 \frac{jL+\frac{j(j-1)}{2}f(0)}{\frac{(j+2)(j+1)}{2}}=\frac{2L+f(0)j-f(0)}{j+3+\frac{2}{j}}$$
 and if $f(0) < 0$, then
 $$\frac{f(z_k)+f(z_{k+1})+\dots f(z_{k+n-1})}{n} \geq
 \frac{jL+\frac{j(j+1)}{2}f(0)}{\frac{(j+2)(j+1)}{2}}=\frac{2Lj}{j^2+3j+2}+\frac{f(0)j}{j+2}.$$
 Again, the both sequences converge to $f(0)$, when $n$ (and thus also $j$) tends to infinity.
By the squeeze theorem, we get that $\frac{f(z_k)+f(z_{k+1})+\dots f(z_{k+n-1})}{n}\xrightarrow{n\to \infty} f(0)$. Moreover, the convergence is uniform with respect to $k$. Thus, by Theorem \ref{t1}, $f(z) \in \widehat{c}$. Finally, by Theorem \ref{texp}, $\widehat{c}\setminus c$ is strongly $\mathfrak{c}$-algebrable.
 
\end{proof}

\begin{theorem}
The family $ S \setminus \widehat{c}$ is strongly $\mathfrak{c}$-algebrable.
\end{theorem}
\begin{proof}
First, put $$(a_n) := (1,2,\frac{3}{2},\frac{5}{4},\frac{7}{4},\frac{9}{8},\frac{11}{8},\frac{13}{8},\frac{15}{8},\dots).$$
Now, put $$(b_n):= (a_1,a_1,a_2,a_1,a_2,a_3,a_1,a_2,a_3,a_4,\dots).$$
Let $z$ be a sequence defined in the following way: put $m_1:=1$ and $m_j:=m_{j-1}+j-1+2^{j-1}$ for $j\in \N$. For $n \in \N$ we define
$$z_n := \left\{ \begin{array}{ccc}
b_j\;\text{ if }\; n \in \{m_{j}, m_{j}+1, \dots m_j+j-1\} \\ 
0 \;\text{ for the remaining }n,%
\end{array}%
\right. $$
that is, $$z=(1,0,0,1,1,0,0,0,0,2,2,2,0,0,0,0,0,0,0,0,1,1,1,1\dots).$$

Let $f\colon\R\to \R$ be an exponential-like function. We will show that $$f(z):=(f(z_1),f(z_2),\dots)\in S \setminus \widehat{c}.$$ By Lemma \ref{lemexp}, we know that every value of $f$ can be obtained only on finitely many arguments. Therefore, there is $k \in\N$ such that $f(a_k) \neq f(0)$. By the construction of $z$, for any $n,N \in \N$ there are $j, m >N$ such that 
$$z_j=z_{j+1}=\dots = z_{j+n-1}= a_k,$$
$$z_m=z_{m+1}= \dots = z_{m+n-1} = 0.$$
Hence
$$ \lim\limits_{n \to \infty}\left(\sup_j \left\{\frac{z_j+\dots+z_{j+n-1}}{n}\right\}\right) \neq \lim\limits_{n \to \infty}\left(\inf_j \left\{\frac{z_j+\dots+z_{j+n-1}}{n}\right\}\right).$$
Thus, $f(z) \notin \widehat{c}.$

Denote by $L$ and $M$, respectively the minimal and the maximal values of $f$ on $[1,2]$. They exist by Weierstrass theorem. Take $j \in \N$ and $n \in \{m_{j},m_j+1,\dots, m_{j+1}-1\}.$ Then, in the finite sequence $(z_1, z_{2}, \dots, z_{n})$, we have at least $2^{j}-2$ zeros and at most $\frac{(j+1)j}{2}$ of other values.
Hence we have 
$$\frac{f(z_1)+f(z_{2})+\dots +f(z_{n})}{n} \leq \frac{M\cdot \frac{(j+1)j}{2}+(2^j-2)f(0)}{m_j}=
 \frac{M\cdot \frac{(j+1)j}{2}+2^jf(0)-2f(0)}{2^j-1+\frac{j(j-1)}{2}}$$
or, if the above number is less than $f(0)$, then 
$$\frac{f(z_1)+f(z_{2})+\dots +f(z_{n})}{n} \leq \frac{M\cdot \frac{(j+1)j}{2}+(2^{j+1}-2)f(0)}{m_{j+1}-1}=
 \frac{M\cdot \frac{(j+1)j}{2}+2^{j+1}f(0)-2f(0)}{2^{j+1}-2+\frac{j(j+1)}{2}}.$$
The both sequences converge to $f(0)$, when $n$ (and thus also $j$) tends to infinity.
Similarly,
 $$\frac{f(z_1)+f(z_{2})+\dots +f(z_{n})}{n} \geq \frac{L\cdot \frac{(j+1)j}{2}+(2^j-2)f(0)}{m_{j}+j-1}=
 \frac{L\cdot \frac{(j+1)j}{2}+2^jf(0)-2f(0)}{2^{j}-1+\frac{j(j+1)}{2}}$$
or, if the above number is greater than $f(0)$, then 
$$\frac{f(z_1)+f(z_{2})+\dots +f(z_{n})}{n} \geq \frac{L\cdot \frac{(j+1)j}{2}+(2^{j+1}-2)f(0)}{m_{j+1}-1}=
 \frac{L\cdot \frac{(j+1)j}{2}+2^{j+1}f(0)-2f(0)}{2^{j+1}-2+\frac{j(j+1)}{2}}.$$
The both sequences converge to $f(0)$, when $n$ (and thus also $j$) tends to infinity.
By the squeeze theorem, we get that $\frac{f(z_1)+f(z_{2})+\dots +f(z_{n})}{n}\xrightarrow{n\to \infty} f(0)$. Thus $f(z) \in S$. Finally, by Theorem \ref{texp}, $S\setminus \widehat{c}$ is strongly $\mathfrak{c}$-algebrable.
 
\end{proof}

\begin{theorem}
The family $ \ell^\infty \setminus S$ is strongly $\mathfrak{c}$-algebrable.
\end{theorem}
\begin{proof}
First, put $$(a_n) := (1,2,\frac{3}{2},\frac{5}{4},\frac{7}{4},\frac{9}{8},\frac{11}{8},\frac{13}{8},\frac{15}{8},\dots).$$
Now, put $$(b_n):= (a_1,a_1,a_2,a_1,a_2,a_3,a_1,a_2,a_3,a_4,\dots).$$
Let $z$ be a sequence defined in the following way: put $m_0:=0$ and $m_j:=j^j$ for $j\in \N$. Let $j \in \N$. For $n \in \{m_{j-1}+1, m_{j-1}+2, \dots ,m_j\}$ we define
$z_n := b_j.$

Let $f\colon\R\to \R$ be an exponential-like function. We will show that $$f(z):=(f(z_1),f(z_2),\dots)\in \ell^\infty \setminus S.$$ Since $f$ is continuous on $[1,2]$ and all values of $z$ are in $[1,2]$, $f(z)$ is bounded. By Lemma \ref{lemexp}, we know that every value of $f$ can be obtained only on finitely many arguments. Therefore, there is $k \in\N$ such that $f(a_k) \neq f(1)$. By the construction of $z$, there are increasing sequences of natural numbers $(j_n)$ and $(l_n)$ such that for any $n\in \N$ 
$$z_{m_{j_n-1}+1}=z_{m_{j_n-1}+2}=\dots = z_{m_{j_n}}= a_k,$$
$$z_{m_{l_n-1}+1}=z_{m_{l_n-1}+2}=\dots = z_{m_{l_n}}= 1.$$

Denote by $L$ and $M$, respectively the minimal and the maximal values of $f$ on $[1,2]$. 
We have
$$\frac{f(1)+f(2)+\dots+f(z_{m_{j_n}})}{m_{j_n}} \leq \frac{M\cdot m_{j_n-1}+f(a_k)\cdot(m_{j_n}-m_{j_{n}-1})}{j_n^{j_n}}$$$$= \frac{M\cdot (j_{n}-1)^{j_{n}-1}+f(a_k)\cdot(j_n^{j_n}-(j_{n}-1)^{j_{n}-1})}{j_n^{j_n}}=\frac{M-f(a_k)}{j_n}\cdot \left(\frac{j_{n}-1}{j_n}\right)^{j_{n}-1}+f(a_k)\xrightarrow{n\to\infty}f(a_k)$$
and 
$$\frac{f(1)+f(2)+\dots+f(z_{m_{j_n}})}{m_{j_n}} \geq \frac{L\cdot m_{j_n-1}+f(a_k)\cdot(m_{j_n}-m_{j_{n}-1})}{j_n^{j_n}}\xrightarrow{n\to\infty}f(a_k).$$
Thus, 
$$\frac{f(1)+f(2)+\dots+f(z_{m_{j_n}})}{m_{j_n}}\xrightarrow{n\to\infty}f(a_k).$$
Similarly,
$$\frac{f(1)+f(2)+\dots+f(z_{m_{l_n}})}{m_{l_n}}\xrightarrow{n\to\infty}f(1).$$
Therefore, the sequence $\left(\frac{f(1)+f(2)+\dots+f(z_{n})}{n}\right)$ is not convergent.
 Thus, $f(z) \notin S$. Finally, by Theorem \ref{texp}, $\ell^\infty\setminus S$ is strongly $\mathfrak{c}$-algebrable.
 
\end{proof} 
From above theorems we can draw the conclusions that the sets $c$ in $\widehat{c}$, $\widehat{c}$ in $S$ and $S$ in $\ell^\infty$ are small in the algebraic sense.
It is worth pointing out that the obtained results are the best possible in the sense of the cardinality of the sets of the generators. 

\section{Measure of the spaces connected to Banach limits}
From now on we will restrict our considerations to subspaces of $X=(-\frac{1}{2},\frac{1}{2})^{\N}$. In $X$ we can consider a product $\sigma$-algebra $\F$ generated by the $\sigma$-algebra of measurable sets in $(-\frac{1}{2},\frac{1}{2})$ and a product measure $\mu$ generated by the Lebesgue measure on $(-\frac{1}{2},\frac{1}{2})$. Then, $(X,\F,\mu)$ is a probability space. This way we can measure how large are the considered families. Similar reasoning was used in \cite{BFN}.
\begin{theorem}\label{m}
\begin{itemize}
\item[(1)] $\mu(S_0\cap X) = 1$,

\item[(2)] $\mu(\widehat{c}\cap X) = 0$.
\end{itemize}
\end{theorem} 
\begin{proof}

Ad (1) For $j \in \N$ define random variables $Y_j: X \to \R$ by the formula $Y_j((x_n)) = x_j$. Then $(Y_j)$ is an i.i.d. sequence. Moreover,
$$EY_j = \int_{-\frac{1}{2}}^{\frac{1}{2}} y dy = 0.$$ 
By the law of large numbers \cite[Theorem 2.25]{Bi},
$$\mu(S_0\cap X)=\mu\left(\left\{x \in X \colon \lim\limits_{n\to \infty} \frac{1}{n} \sum_{j=1}^n x_j = 0\right\}\right)=\mu\left(\left\{x \in X \colon \lim\limits_{n\to \infty} \frac{1}{n} \sum_{j=1}^n Y_j(x) = 0\right\}\right) = 1 .$$

Ad (2) We have 
$$\widehat{c} \cap X = \left\{ x \in X \colon \exists_{s\in[-\frac{1}{2},\frac{1}{2}]} \forall_{k \in \N} \exists_{N \in \N} \forall_{j \in \N} \forall_{n \geq N} \left\vert \frac{x_j+\dots+x_{j+n-1}}{n} - s \right\vert < \frac{1}{k} \right\} 
$$$$= \bigcup_{s\in[-\frac{1}{2},\frac{1}{2}]} \bigcap_{k \in \N} \bigcup_{N \in \N} \bigcap_{j \in \N} \bigcap_{n \geq N} \left\{ x \in X \colon \frac{x_j+\dots+x_{j+n-1}}{n} \in \left(s- \frac{1}{k},s+\frac{1}{k}\right) \right\} 
$$$$\subset \bigcup_{N \in \N} \bigcap_{j \in \N} \bigcap_{n \geq N} \left\{ x \in X \colon \frac{x_j+\dots+x_{j+n-1}}{n} \in \left[-\frac{1}{2},\frac{1}{4}\right) \right\} \cup$$$$\cup  \bigcup_{N \in \N} \bigcap_{j \in \N} \bigcap_{n \geq N} \left\{ x \in X \colon \frac{x_j+\dots+x_{j+n-1}}{n} \in \left(-\frac{1}{4},\frac{1}{2}\right] \right\}.$$
Take $N \in \N$. Consider the set $\bigcap_{j \in \N} \bigcap_{n \geq N} \left\{ x \in X \colon \frac{x_j+\dots+x_{j+n-1}}{n} \in \left[-\frac{1}{2},\frac{1}{4}\right) \right\}$. Observe that 
$$\bigcap_{j \in \N} \bigcap_{n \geq N} \left\{ x \in X \colon \frac{x_j+\dots+x_{j+n-1}}{n} \in \left[-\frac{1}{2},\frac{1}{4}\right) \right\} \subset \bigcap_{m \in \N} \left\{ x \in X \colon \frac{x_{(m-1)N+1}+\dots+x_{mN}}{N} \in \left[-\frac{1}{2},\frac{1}{4}\right) \right\}.$$
Define random variables $Y_j: X \to \R$ by the formula $Y_m((x_n)) = \frac{x_{(m-1)N+1}+\dots+x_{mN}}{N}$. Observe that $(Y_m)$ is an i.i.d. sequence. Put 
$$p = \mu\left(Y_1 < \frac{1}{4}\right) \in (0,1).$$ 
Then, using the independence of $(Y_m)$, we get $$\mu\left( \bigcap_{m \in \N} \left\{ x \in X \colon \frac{x_{(m-1)N+1}+\dots+x_{mN}}{N} \in \left[-\frac{1}{2},\frac{1}{4}\right) \right\} \right) = \mu\left(Y_1 < \frac{1}{4}, Y_2 < \frac{1}{4}, \dots\right)$$$$ = \mu\left(Y_1 < \frac{1}{4}\right)\cdot\mu\left(Y_2 < \frac{1}{4}\right)\cdot \dots = p \cdot p \cdot \dots = 0.$$
Hence also $$\mu\left(\bigcap_{j \in \N} \bigcap_{n \geq N} \left\{ x \in X  \colon \frac{x_j+\dots+x_{j+n-1}}{n} \in \left[-\frac{1}{2},\frac{1}{4}\right) \right\}\right) = 0.$$ A countable union of null sets is a null set, so
$$\mu\left(\bigcup_{N \in \N} \bigcap_{j \in \N} \bigcap_{n \geq N} \left\{ x \in X \colon \frac{x_j+\dots+x_{j+n-1}}{n} \in \left[-\frac{1}{2},\frac{1}{4}\right) \right\}\right) = 0.$$
Similarly we show that 
$$\mu\left(\bigcup_{N \in \N} \bigcap_{j \in \N} \bigcap_{n \geq N} \left\{ x \in X \colon \frac{x_j+\dots+x_{j+n-1}}{n} \in \left(-\frac{1}{4},\frac{1}{2}\right] \right\}\right) = 0.$$
Finally, $\mu(\widehat{c} \cap X) = 0$.
\end{proof}
\begin{corollary}
$\mu(S\cap X) = 1, \mu(c_0\cap X)=\mu(c\cap X) = \mu(\widehat{c_0} \cap X) = 0.$
\end{corollary}
\begin{remark}
We will get the same results if we consider $X= (-a,a)^\N$ for $a >0$ with a product measure generated by the uniform distribution on $(-a,a)$.
\end{remark}
\begin{remark}
In the space $\R^{\N}$ we could consider a product measure generated by a normal distribution (with expected values equal to $0$) on $\R$. But then, the measure of $\ell^{\infty}$ would be equal to $0$. Indeed,
$$\mu(\ell^{\infty}) = \mu\left(\bigcup_{n \in \N} \left\{(x_i) \in \R^{\N} \colon \forall_{i \in \N} \; x_i \in (-n,n)\right\}\right) = 0.$$
However, we would still be able to show an analogue of Theorem \ref{m}(1).
\end{remark}

\section*{Acknowledgments}
The author would like to express his gratitude for Professor Jacek Jachymski for the inspiration and sugestions of simplifying some of the proofs, and for Professor Marek Balcerzak for many valuable remarks and advices.

\end{document}